\documentclass{article}

\usepackage{titlesec, blindtext, color}
\usepackage{tikz-cd}

\usepackage{mathtools}
\usepackage{amsmath}
\usepackage{amsfonts}
\usepackage{amssymb}
\usepackage{amsthm}
\usepackage{graphicx}
\usepackage{dsfont}
\usepackage{hyperref}
\usepackage{mathabx}
\usepackage{tikz}
\usepackage{comment}
\usepackage{float}

\usetikzlibrary{calc,decorations.pathmorphing,shapes}

\usepackage{url}
\makeatletter
\g@addto@macro{\UrlBreaks}{\UrlOrds}
\makeatother
\usepackage{multicol}
\usepackage{lineno}
\usepackage{makeidx}

\theoremstyle{definition}
\newtheorem{defi}{Definition}[section]
\newtheorem{thm}[defi]{Theorem}
\newtheorem{lem}[defi]{Lemma}
\newtheorem{cor}[defi]{Corollary}

\newtheorem{quest}[defi]{Question}
\newtheorem{claim}[defi]{Claim}

\newtheorem{fact}[defi]{Fact}

\DeclareMathOperator{\cov}{cov}

\DeclareMathOperator{\cof}{cof}

\DeclareMathOperator{\Bor}{Bor}

\title{Borel chromatic numbers of closed graphs and forcing with uniform trees}
\author{Michel Gaspar and Stefan Geschke}
\date{}

\begin{document}

\maketitle

\begin{abstract}
    In this work we continue the tradition initiated in \cite{geschke2011} of viewing the uncountable Borel chromatic number of analytic graphs as cardinal invariants of the continuum. We show that various uncountable Borel chromatic numbers of closed graphs can be consistently different, as well as consistently equal to the continuum. This is done using arguments that are typical to Axiom A forcing notions.
\end{abstract}

\section{Introduction}

A \textit{graph} on a set $X$ is a symmetric irreflexive relation $G \subseteq X^2$. A set $A \subseteq X$ is \textit{$G$-independent} iff $A^2 \cap G = \emptyset$. The \textit{chromatic number} of $G$, $\chi(G)$, is the least cardinality of a family of $G$-independent sets covering the underlying space $X$. 

Now if $X$ is endowed with a Polish topology, we may look at $G$-independent sets that have additional complexity (e.g., Borel, compact, Baire measurable etc). In this case, the \textit{Borel chromatic number} of $G$, $\chi_B (G)$, is the least cardinality of a family of $G$-independent Borel sets covering the space $X$. These numbers were extensively studied in \cite{kechris1999borel} from a ZFC standpoint. We are only interested in graphs with \textit{uncountable} Borel chromatic number.

The study of uncountable Borel chromatic numbers as cardinal invariants of the continuum was proposed in \cite{geschke2011}. In his approach, the author proves that it is consistent to have the continuum arbitrarily big, while many definable graphs have Borel chromatic number at most $\aleph_1$. Namely, there is a ccc extension of the universe such that any closed graph $G$ on a Polish space $X$ has Borel chromatic number at most $\aleph_1$ --- provided it has no perfect cliques --- while the continuum is arbitrarily big. Next we introduce two such graphs that play an important role in this work: $G_0$ and $G_1$.

Let $(s_k)_{k \in \omega}$ be a sequence of elements of $2^{<\omega}$ such that
\begin{itemize}
    \item $|s_k| = k$, and
    \item for every $s \in 2^{<\omega}$, there is $k \in \omega$ such that $s_k \supseteq s$.
\end{itemize}
Then $G_0$ is the graph on $2^\omega$ defined as
\[G_0 = \{(s_k^\frown i^\frown x, s_k^\frown (1-i)^\frown x)\ |\ k \in \omega \text{ and } x \in 2^\omega\}.\]
This is a closed acyclic graph with uncountable Borel chromatic number. In fact, it follows from the proof of Proposition 6.2 of \cite{kechris1999borel} that any Baire measurable $G_0$-independent subset of $2^\omega$ is meager. Hence, 

\begin{equation}\label{covMG0}
 \cov(\mathcal{M}) \leq \chi_B (G_0).   
\end{equation}

The \textit{$G_0$-dichotomy} (see Fact \ref{g0dich}) says this is the least possible uncountable Borel chromatic number for an analytic graph. In effect, if $G$ is an analytic graph on a Polish space $X$, then either $\chi_B (G) \leq \aleph_0$ or there exists a continuous homomorphism from $G_0$ to $G$. Ben Miller in \cite{miller2012graph} showed that this dichotomy implies many well-known descriptive set-theoretic dichotomies (e.g., the perfect set property, or Suslin dichotomy about coanalytic relations).

It was briefly mentioned in \cite{kechris1999borel} that a measure-analog of inequality \ref{covMG0} would hold for $G_0$. Namely, that $\chi_B (G_0) \geq \cov(\mathcal{N})$. However, this is actually open. In fact, it is not true that Lebesgue measurable $G_0$-independent sets have Lebesgue measure zero: using Theorem 3.3 of \cite{miller2008measurable} one easily gets a $G_0$-independent positive-measure $F_\sigma$ subset of $2^\omega$ (for more, see discussion after Question \ref{questrandom}). On the other hand, for the more homogeneous graph $G_1$, the measure-analog of the fact above holds. 

Let $G_1$ be the graph on $2^\omega$ defined by
\[G_1 = \{(x, y)\ |\ \exists ! n (x(n) \neq y(n))\}.\]
Since any Lebesgue measurable $G_1$-independent subset of $2^\omega$ has measure zero, then

\begin{equation}\label{covNG1}
\cov(\mathcal{N}) \leq \chi_B (G_1).  
\end{equation}

Now recall the definiton of \textit{Vitali's equivalence relation} $E_0$:
\[x E_0 y \leftrightarrow \forall^\infty n (x(n) = y(n)).\]
This is the least non-smooth Borel equivalence relation and $G_0$ can be seen as the graph-analog of $E_0$; and the \textit{Glimm-Effros dichotomy} (see Theorem 1.1. of \cite{harrington1990glimm}) is the analog of the $G_0$-dichotomy: it says that $E_0$ embeds continuously into any non-smooth (i.e., not Borel reducible to the identity) Borel equivalence relation on a Polish space. 

Let $\chi_B (E_0)$ denote the least cardinality of a family of Borel $E_0$-transversals covering $2^\omega$ --- we shall refer to it as the Borel chromatic number of $E_0$. From the observation that the connected components of $G_0$ and $G_1$ are equivalence classes of $E_0$, it readily follows that 

\begin{equation}\label{G0toE0}
\chi_B (G_0) \leq \chi_B (G_1) \leq  \chi_B (E_0).
\end{equation}

It follows from \cite{geschke2011} that there is a $\sigma$-centered extension of the universe such that $\chi_B (E_0) = \aleph_1$. Furthermore, in this extension any locally countable $F_\sigma$-graph has Borel chromatic number at most $\aleph_1$.

Finally, we have

\begin{equation}\label{rG1}
    \chi_B (G_1) \leq \mathfrak{r},
\end{equation}

and

\begin{equation}\label{ramsey}
\mathfrak{h} \leq \chi_B (G_1),
\end{equation}

where $\mathfrak{r}$ is the \textit{reaping number} and $\mathfrak{h}$ is the \textit{distributivity} (or \textit{shattering}) \textit{number} (see the correspondent sections about $\mathfrak{r}$ and $\mathfrak{h}$ in Chaper 8 of \cite{halbeisen2012combinatorial}).

The inequality (\ref{rG1}) is implicit in the proof of Lemma 3 of \cite{brendle1995}. As for $(\ref{ramsey})$, let $r^0$ denote the $\sigma$-ideals of \textit{completely Ramsey-null sets}. This may be seen as the $\sigma$-ideal subsets of $[\omega]^\omega$ that are meager in the \textit{Ellentuck topology} (see, e.g., 1.3. of \cite{Brendle1994StrollingTP}).

We know that $\mathfrak{h} \leq \cov(\Bor([\omega]^\omega) \cap r^0)$ (see, e.g., Theorem 9.2 of \cite{halbeisen2012combinatorial}). Moreover, \cite{halbeisen2003making} refined the Ellentuck topology, defining the \textit{doughnut topology}, and proved that there exists a correspondence between the Borel sets that are meager in this topology and Borel $G_1$-independent sets (see Section 0 and Fact 1.4. of \cite{halbeisen2003making}). 

Combining (\ref{G0toE0}), (\ref{rG1}) and (\ref{ramsey}) together with known inequalities between cardinals from \textit{van Douwen's} diagram (see, e.g., \cite{khomskii2012regularity}), we obtain the following diagram:

\begin{figure}[H]
    \centering
    \includegraphics[scale=0.60]{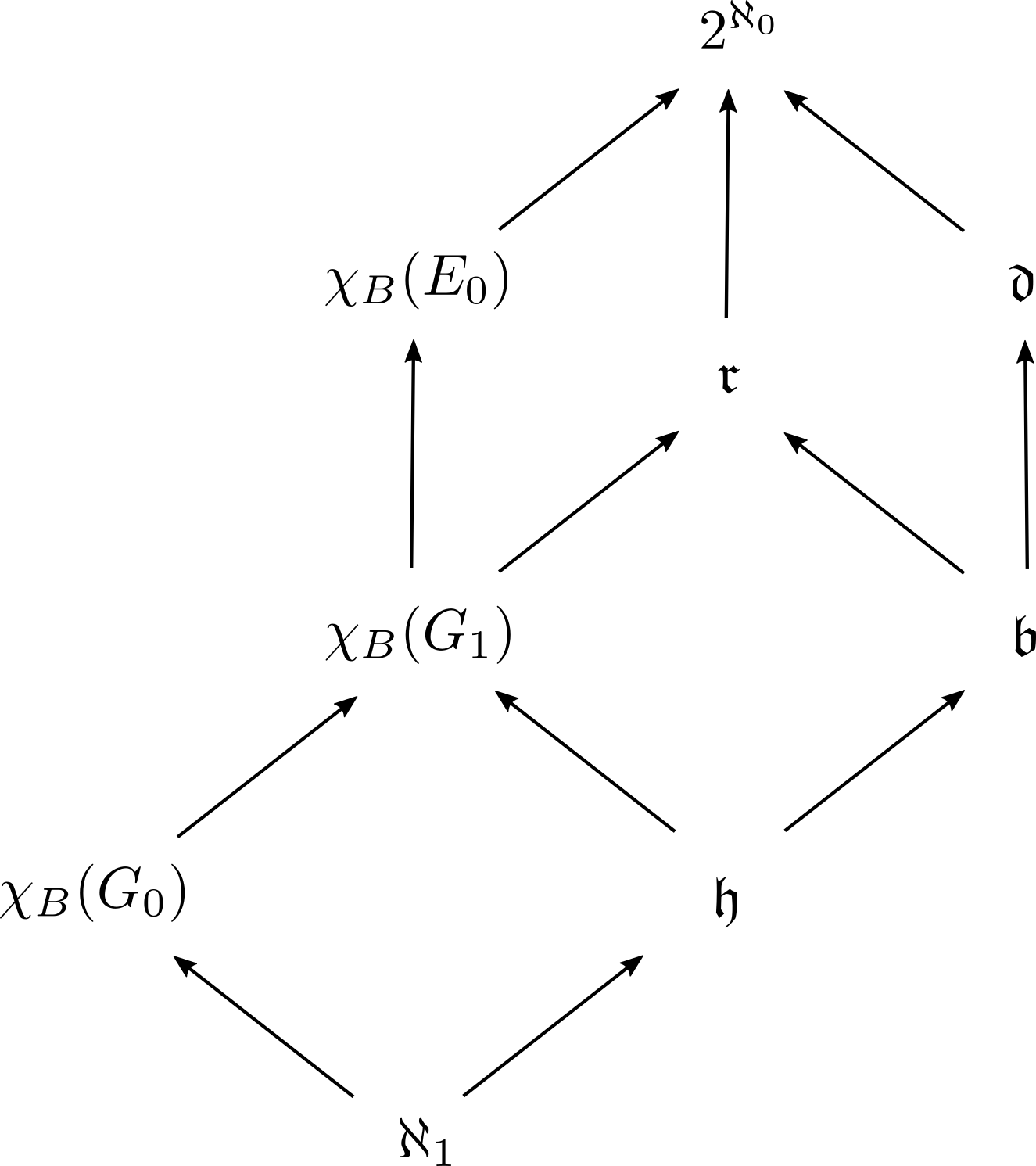}
    \caption{Borel chromatic numbers added to van Douwen's diagram}
    \label{BCNdiagram}
\end{figure}

The ultimate goal is to show the consistency of the strict inequality between any two cardinals in this diagram. 

For our porpuses, the inequalities (\ref{covMG0}) and (\ref{covNG1}) are omitted from the above diagram. The consistency of $\chi_B (G_0) > \mathfrak{d}$ will imply $\chi_B (G_0) > \cov(\mathcal{M})$. The consistency of $\chi_B (G_0) < \cov(\mathcal{N})$ is open (see Question \ref{questrandom}).

As consequences of our methods, we prove the consistencies of $\chi_B (G_0) < \chi_B (G_1)$, $\chi_B (G_1) < \chi_B (E_0)$, and $\chi_B (G_0) > \mathfrak{d}$.

Part of this work, namely Theorem \ref{silverthm} item a. was independently solved by \cite[Corollary 3.49]{zapletal2004descriptive}, as well as a slightly weaker version of item b. (see  \cite[Corollary 3.38]{zapletal2019hypergraphs}). The methods used in there are, however, very indirect and completely different from the ones used here.

\section{Preliminaries: forcing notions of perfect trees}

Recall that a forcing notion $\mathbb{P}$ satisfies \textit{Axiom A} if there exists $(\leq_n)_{n \in \omega}$, a non-$\subseteq$-decreasing sequence of partial orders of $\mathbb{P}$, with the following properties:
    \begin{itemize}
        \item[(1)] if $(p_n)_{n \in \omega}$ is a sequence such that $p_{n+1} \leq_n p_n$, for all $n \in \omega$, then there exists $q \in \mathbb{P}$ such that $q \leq_n p$, for all $n \in \omega$; and
        
        \item[(2)] if $A \subseteq \mathbb{P}$, $p \in \mathbb{P}$ and $n \in \omega$, then there exists $q \leq_n p$ compatible with at most countably many elements of $A$.
    \end{itemize}
A sequence $(p_n)_{n \in \omega}$  as in item (1) is called a \textit{fusion sequence}. If $q$ in item (2) can be chosen to be compatible with at most \textit{finitely} many elements of $A$, we say that $\mathbb{P}$ satisfies the \textit{strong Axiom A}.

It is well-known that Axiom A forcing notions are proper. Moreover, strong Axiom A forcing notions are additionally $\omega^\omega$-bounding (see, e.g., Corollary 2.1.12 from Theorem 2.1.4 of \cite{roslanowski1998norms}).

One very appealing aspect of Axiom A forcing notions is how convenient they are when working with countable support iterations: 

For an ordinal $\alpha \geq 1$, let $\mathbb{P}_{\alpha}$ denote the $\alpha$-iteration of $\mathbb{P}$ with countable support.

Let $F$ be a finite subset of $\alpha$ and $\eta: F \rightarrow \omega$. Say that $q \leq_{F, \eta} p$ iff  
\[\forall \gamma \in F \left(q \restriction \gamma \Vdash q(\gamma) \leq_{\eta(\gamma)} p(\gamma)\right).\]

A sequence $(p_n)_{n \in \omega}$ in $\mathbb{P}_\alpha$ is a \textit{(iterated) fusion sequence} iff there exist sequences $(F_n)_{n \in \omega}$, and $(\eta_n)_{n \in \omega}$ such that, for every $n \in \omega$,
    
    \begin{itemize}
        \item[(1)] $F_n \subseteq \alpha$ is finite; 
        \item[(2)] $F_n \subseteq F_{n+1}$;
        \item[(3)] $\eta_n: F_n \rightarrow \omega$;
        \item[(4)] $\eta_n (\gamma) \leq \eta_{n+1} (\gamma)$, for all $\gamma \in F$;
        \item[(5)] $p_{n+1} \leq_{F_n, \eta_n} p_n$; and
        \item[(6)] for all $\gamma \in \mathrm{supp}(p_n)$, there is $m \in \omega$ such that $\gamma \in F_m$ and $\eta_m (\gamma) \geq n$.
    \end{itemize}
    
The fusion $q$, of a fusion sequence $(p_n)_{n \in \omega}$, is defined recursively such that
\[\forall \gamma < \alpha \left(q \restriction \gamma \Vdash q(\gamma) \text{ is the fusion of } (p_n (\gamma)\right))_{n \in \omega}.\]

We shall introduce the two main examples of (strong) Axiom A forcing notions that will be used throughout the text: the $E_0$ and the \textit{Silver} forcing notions. Before this, we need to fix some notation:

The simplest example of strong Axiom A is the Sacks forcing: 

A tree $p \subseteq 2^{<\omega}$ is a \textit{Sacks tree} iff it is a \textit{perfect tree} --- i.e.,  for every $t \in p$, there exists $s \supseteq t$ such that $s^\smallfrown 0, s^\smallfrown 1 \in p$. 

The \textit{Sacks forcing}, $\mathbb{S}$, consists of Sacks trees ordered by direct inclusion. 

For a perfect tree $p \subseteq 2^{<\omega}$, let $\mathrm{st}(p)$ denote its stem; and $\mathrm{spl}(p)$ denote its set of splitting nodes. There exists a natural bijection there exists a bijection $\sigma \mapsto \sigma^*$ from $2^{<\omega}$ to $\mathrm{spl}(p)$, described as follows:
$\emptyset^* = \mathrm{st}(p)$; and $(\sigma^\smallfrown i)^* \in \mathrm{spl} (p)$ is the minimal splitting node of $p$ such that $(\sigma^\smallfrown i)^* \supsetneq {\sigma^*}^\smallfrown i$, for each $i < 2$.

Let $L_n (p) = \left\{\sigma^* \ |\ \sigma \in 2^n\right\}$ denote the \textit{$n$-th splitting level of $p$}, for each $n \in \omega$; and for $\sigma \in 2^{<\omega}$, let $p \ast \sigma = \left\{s \in p\ |\ \sigma^* \subseteq s \text{ or } s \subseteq \sigma^*\right\}$ denote the \textit{restriction of $p$ to $\sigma$}. We may also denote the restriction of $p$ to one of its splitting nodes $t \in p$ by $p_{t}$ (i.e., for $\sigma \in 2^{<\omega}$, $p_{\sigma^*} = p \ast \sigma$).

A sequence $(\leq_n)_{n \in \omega}$ of partial orders of $\mathbb{S}$ is defined as follows: for every $n \in \omega$,
\[q \leq_n p \leftrightarrow q \leq p \text{ and } L_n (q) = L_n (p).\]

This way, if $(p_n)_{n \in \omega}$ is a fusion sequence witnessed by $(\leq_n)_{n \in \omega}$, then its fusion is $q = \bigcap_{n \in \omega} p_n$.

Now we introduce the two main examples of strong Axiom A that will be considered throughout this text. These are forcing notions of perfect trees for which the sequence $(\leq_n)$, restricted to them, witnesses strong Axiom A. 

Last, let $[p]$ denote the set of branches throught $p$ --- i.e., of all $x \in 2^{\omega}$ such that $x \restriction n \in p$, for all $n \in \omega$.

A tree $p \subseteq 2^{<\omega}$ is an \textit{$E_0$-tree} iff it is perfect; and for every $s \in \mathrm{spl}(p)$, there are $s_0 \supseteq s^\smallfrown 0$ and $s_1 \supseteq s^\smallfrown 1$, of the same length, such that
\[\left\{x \in 2^\omega\ |\ s_0^\smallfrown x \in [p]\right\} = \left\{x \in 2^\omega\ |\ s_1^\smallfrown x \in [p]\right\}.\]
The \textit{$E_0$-forcing}, $\mathbb{E}_0$, consists of $E_0$-trees ordered by inclusion.

A tree $p \subseteq 2^{<\omega}$ is a \textit{Silver tree} iff it is an $E_0$-tree such that, for every $s \in \mathrm{spl}(p)$,
\[\left\{x \in 2^\omega\ |\ s^\smallfrown 0^\smallfrown x \in [p]\right\} = \left\{x \in 2^\omega\ |\ s^\smallfrown 1^\smallfrown x \in [p]\right\}.\]
The \textit{Silver forcing}, $\mathbb{V}$, consists of Silver conditions ordered by inclusion.

As said earlier, $(\leq_n \restriction\ \mathbb{E}_0)_{n \in \omega}$ witnesses strong Axiom A for $\mathbb{E}_0$; and $(\leq_n \restriction\ \mathbb{V})_{n \in \omega}$ witnesses strong Axiom A for $\mathbb{V}$. 

The \textit{$\mathbb{E}_0$-model} is the generic extension obtained by forcing with an $\omega_2$-iteration of $\mathbb{E}_0$, with countable support, over a model of CH. The \textit{Silver model} is defined similarly.

\section{Separating Borel chromatic numbers}

In this section we tackle the problem of the consistency of $\chi_B (G_1) < \chi_B (E_0)$ and $\chi_B (G_0) < \chi_B (G_1)$. 

The next fact shows $\mathbb{E}_0$ and $\mathbb{V}$ increase $\chi_B (E_0)$ and $\chi_B (G_1)$, respectively. For this, let $I_{E_0}$ and $I_{G_1}$ denote the $\sigma$-ideals generated by Borel $E_0$-transversals, and by Borel $G_1$-independent sets, respectively.

\begin{fact}[Lemmas 2.3.29 and 2.3.37 of \cite{zapletal2004descriptive}]\label{zapsilver} Let $A \subseteq 2^\omega$ be an analytic set. 
\begin{itemize}
    \item[(a)] If $A \notin I_{E_0}$, then it contains the set of branches of an $E_0$-tree.
    \item[(b)] If $A \notin I_{G_1}$, then it contains the set of branches of a Silver tree.
\end{itemize}
\end{fact}
It follows that the map $p \mapsto [p]$ is a dense embedding from the forcing to the poset of Borel positive (with respect to the ideal) sets, which is a forcing notion known to increase the Borel chromatic number of the graph. 

\begin{thm}\label{silverthm} Let $G$ be a closed graph defined on a compact Polish space $X$. 

\begin{itemize}
    \item[(a)]  If $G$ has no perfect cliques, then every point of $X$ in the $E_0$-extension is contained in a compact $G$-independent set coded in the ground-model; and
    \item[(b)] if $G$ has no $4$-cycles, then every point of $X$ in the Silver extension is contained in a compact $G$-independent set coded in the ground-model.
\end{itemize}
In each of the above cases, $\chi_B (G) \leq |2^{\aleph_0} \cap V|$.
\end{thm}

\begin{cor}
It is consistent with ZFC that $\chi_B (G_1) < \chi_B (E_0)$ and $\chi_B (G_0) < \chi_B (G_1)$.
\end{cor}

Our strategy is to divide the proof in successor and limit steps of iterations. For this, it will be crucial to tackle the case of adding a single generic real.

It will be slightly more convenient to work with graphs defined on $\omega^\omega$ rather than arbitrary Polish spaces. This can be done since any perfect Polish space is the continuous injective image of $\omega^\omega$. If $X$ is a perfect Polish space, $f$ is a continuous injection such that $f[\omega^\omega] = X$, and $G$ is a closed graph on $X$, then the \textit{pull-back of $G$ by $f$}, defined by
\[f^* [G] = \left\{\left(f^{-1} (x),  f^{-1} (y)\right) \in X^2\ |\  (x, y) \in G\right\},\]
is a graph on $X$. Moreover, $G$ has perfect clique iff $f^* [G]$ has a perfect clique; and it has a cycle iff $f^*[G]$ has a cycle.

We first will inspect when the generic real is contained in a compact $G$-independent set coded in the ground model, for $G$ closed on $2^\omega$. 

Let $p$ be an $E_0$-tree and $G$ be a closed graph on $2^\omega$. Say that $q \leq p$ \textit{agrees with} $G$ iff
\[([q \ast 0] \times [q \ast 1]) \cap E_0 \subseteq G.\]
This is similarly defined for Silver trees replacing $E_0$ with $G_1$.

Let $D(G)$ be the set of conditions agreeing with $G$, for either $\mathbb{E}_0$ or $\mathbb{V}$ --- i.e.,
\[D(G) = \{q\ |\ q \text{ agrees with } G\}.\]

\begin{lem}\label{lemdense} Let $G$ be a closed graph on $2^\omega$. Considering either $E_0$ or Silver trees: there exists $q \leq p$ such that $[q]$ is $G$-independent iff $D(G)$ is not dense below $p$.
\end{lem}

\begin{proof}

Trivially, if $q \leq p$ is such that $[q]$ is $G$-independent, then no stronger $r \leq q$ agrees with $G$, which implies that $D(G)$ is not dense below $p$. 

For the converse, assume $D(G)$ is not dense below $p$. We shall construct a fusion sequence $(p_n)_{n \in \omega}$ such that, for every $n \in \omega$ and every $\sigma \in 2^n$,
    \begin{itemize}
        \item[(1)] no $r \leq p_n$ agrees with $G$; and
        \item[(2)] $([p_n \ast \sigma^\smallfrown 0] \times [p_n  \ast \sigma^\smallfrown 1]) \cap G = \emptyset$.
    \end{itemize}

Let $p_0 \leq p$ be such that no stronger condition agrees with $G$. Assume $p_n$ and let $\left\{\sigma_0, ..., \sigma_{m-1}\right\}$ be an enumeration of $2^{n+1}$. We define $p_{n+1}$ in $2^{n+1}$ steps using successive amalgamation: 

Let $q_0 = p_n$ and assume $q_{j-1}$ is already defined for all $j < m-1$. By induction, the condition $q_{j-1} \ast \sigma_j$ does not agree with $G$.  Hence, there exists $(z_0, z_1) \in ([q_{j-1} \ast \sigma_j^\frown 0] \times [q_{j-1} \ast \sigma_j^\frown 1]) \cap E_0$ (or $G_1$, in case of forcing with Silver trees) such that $(z_0, z_1) \notin G$. Since $G$ is a closed graph, let $s_0 \subseteq z_0$ and $s_1 \subseteq z_1$ be such that 
\[([s_0] \times [s_1]) \cap G = \emptyset.\]

Let $q_j \leq_n q_{j-1}$ be such that $\mathrm{st}(q_j \ast \sigma_j^\frown 0) \supseteq s_0$, and $\mathrm{st}(q_j \ast \sigma_j^\frown 1) \supseteq s_1$. Let $p_{n+1} = q_{m-1}$. Then $q = \bigcap_{n \in \omega} p_n$ is a condition such that $[q]$ is a $G$-independent. \qedhere
\end{proof}




\begin{lem}\label{mainlemma1} Let $G$ be a closed graph on $2^\omega$. If $D(G)$ is dense below $p$, 
\begin{itemize} 
    \item[(a)] for the $E_0$ forcing, if $G$ has a perfect clique; and
    \item[(b)] for the Silver forcing, if $G$ has a $4$-cycle.
\end{itemize}
\end{lem}

\begin{proof} 

For (a), using the fact that $D(G)$ is dense below $p$ we will construct a fusion sequence $(p_n)_{n \in \omega}$ of perfect trees (but not necessarily $E_0$-trees) such that, for all $n \in \omega$ and all $\sigma \in 2^n$,
\[[p_n \ast \sigma^\smallfrown 0] \times [p_n \ast \sigma^\smallfrown 1] \subseteq G.\]
Then $q = \bigcap_{n \in \omega} p_n$ will be a perfect tree such that $[q]$ is $G$-clique. 

In order to construct such sequence, simply note that if $r \leq p$ is a condition that agrees with $G$, it will follow from the closedness of $G$ that $[r \ast 0] \times [r \ast 1] \subseteq G$: for every $(z, w) \in [r \ast 0] \times [r \ast 1]$, there exists a sequence $(w_n)_{n \in \omega}$ in $[r \ast 1]$ such that $(z, w_n) \in E_0$ and $(w_n)_{n \in \omega}$ converges to $w$. Since $r$ agree with $G$, then $(z, w_n) \in G$, for every $n \in \omega$ and, since $G$ is a closed graph, $(z, w) \in G$. 

For (b), let $r \leq p$ be a condition that agrees with $G$.

For $z \in [r \ast 0]$, let $z'$ denote the copy of $z$ in $[r \ast 1]$ --- i.e., $(z, z') \in G_1$.

If there exists $(z, w) \in [r \ast 0]^2 \cap G$ such that $(z', w') \in G$, it follows from $(z, z'), (w, w') \in ([r \ast 0] \times [r \ast 1]) \cap G_1 \subseteq G$ that $\{z, z', w, w'\}$ is a $4$-cycle. If this fails, then for all $(z, w) \in [r \ast 0]^2$, 
\[(z, w) \in G \leftrightarrow (z', w') \notin G.\]
It follows from Ramsey theorem that in any set $\left\{z_0, ..., z_{R(4) - 1}\right\} \subseteq [r \ast 0]$ of $R(4)$ vertices (where $R(4)$ denotes the Ramsey number of $4$), either there exists an $4$-cycle; or there exists an independent subset of size $4$, which implies that $\left\{z'_0, ..., z_{R(4) - 1}'\right\} \subseteq [r \ast 1]$ has a $4$-cycle. \qedhere
\end{proof}

From Lemmas \ref{lemdense} and \ref{mainlemma1}, it follows that the $E_0$-real is contained in a compact $G$-independent set coded in the ground model, if $G$ is a closed graph on $2^\omega$ without perfect cliques; and the Silver real is contained in a $G$-independent set coded in the ground model, if $G$ is a closed graph on $2^\omega$ without $4$-cycles. We now need to show that this happens for any other element of $\omega^\omega$ added by $\mathbb{E}_0$ and $\mathbb{V}$, respectively. This is where the minimality of these forcing notions may play a role:

\begin{fact} Let $p \in \mathbb{E}_0$ and $f: [p] \rightarrow \omega^\omega$ be a continuous function. Then there exists $q \leq p$ such that $f \restriction [q]$ is either constant or injective. In particular, $\mathbb{E}_0$ adds reals of minimal degree. 

The same holds true replacing $\mathbb{E}_0$ with $\mathbb{V}$. 
\end{fact}

\begin{proof}
Lemma 4.7 and Proposition 4.8 of \cite{grigorieff1971combinatorics}
\end{proof}

Only the proof for $\mathbb{V}$ is found in \cite{grigorieff1971combinatorics}, as the forcing $\mathbb{E}_0$ was only introduced in \cite{zapletal2004descriptive}. However, the same proof can be easily adapted to work for $\mathbb{E}_0$ as well.

\begin{lem}\label{mainlemma2} Let $\dot{x}$ be a name for an element of $\omega^\omega$, witnessed by $p$, and $f:[p] \rightarrow \omega^\omega$ a continuous function in the ground model such that $p \Vdash f(\dot{x}_\mathrm{gen}) = \dot{x}$ and $f$ is either constant or injective. If $G$ is a closed graph on $\omega^\omega$, then there exists $q \leq p$ such that $f[q]$ is $G$-independent
\begin{itemize}
    \item[(a)] for the the $E_0$-forcing, if $G$ has no perfect cliques; and
    \item[(b)] for the Silver forcing, if $G$ has no $4$-cycles.
\end{itemize}
\end{lem}

\begin{proof}

Assume, without loss of generality, that $f$ is injective and let $f^*[G]$ be the the pull-back of $G$ by $f$. In both cases, (a) and (b), we find $q \leq p$ such that $[q]$ is $f^*[G]$-independent. This implies that $f[q]$ is a $G$-independent compact set. \qedhere
\end{proof}

Note that the image of the function $f \restriction q$ is the set $[T_q (\dot{x})]$, where 
\[T_q (\dot{x}) = \{s \in \omega^{<\omega}\ |\ \exists r \leq q (q \Vdash s \subseteq \dot{x})\}.\]
This is called the \textit{tree of $q$-possitibilities for $\dot{x}$}. 

We will tackle the iteration by introducing a suitable notion of \textit{faithfulness}. 

Let $\alpha \geq 1$ be any ordinal and $p$ be an $\alpha$-iterated ($\mathbb{E}_0$ or $\mathbb{V}$) condition. For finite $F \subseteq \alpha$, $\eta: F \rightarrow \omega$ and $\sigma \in \prod_{\gamma \in F} 2^{\eta(\gamma)}$, let $p \ast \sigma$ be defined such that
\[\forall \gamma \in F \left((p \ast \sigma) \restriction \gamma \Vdash (p \ast \sigma) (\gamma) = p(\gamma) \ast \sigma(\gamma) \right).\]

Let $\dot{x}$ be a name for an element of $\omega^\omega$, which is not added at any proper stage of the iteration, witnessed by $p$. 

A condition $q \leq p$ is \textit{$G$-$(F, \eta)$-faithful} iff 
    \[([T_{q \ast \sigma} (\dot{x})] \times [T_{q \ast \tau} (\dot{x})]) \cap G = \emptyset,\]
    for all distinct $\sigma, \tau \in \prod_{\gamma \in F} 2^{\eta(\gamma)}$.

Let $\eta'(\gamma) = \eta(\gamma)$, for all $\gamma \neq \beta$, and $\eta' (\beta) = \eta(\beta) + 1$. 

\begin{lem}\label{bookkeeping} Let $G$ be a closed graph and $q \leq_{F, \eta} p$ be a $G$-$(F, \eta)$-faithful condition. Then there exists a $G$-$(F, \eta')$-faithful condition $r \leq_{F, \eta'} q$,
\begin{itemize}
    \item[(a)] for countable support iterations of the $\mathbb{E}_0$-forcing; if $G$ has no perfect cliques; and
    \item[(b)] for countable support iterations of $\mathbb{V}$, if $G$ has no $4$-cycles.
\end{itemize}
\end{lem}

\begin{proof} The proof works exactly the same both for $\mathbb{E}_0$ and $\mathbb{V}$, using either item (1) or (2) of Lemma \ref{mainlemma2}, at successor steps. Let $\{\sigma_0, ..., \sigma_{m-1}\}$ be an enumeration of $\prod_{\gamma \in F} 2^{\eta(\gamma)}$.

    \begin{itemize}
        \item Case 1: $\alpha$ is limit. 
        
        We define a $\leq_{F, \eta}$-decreasing sequence $(q_j)_{j < m}$, along with names for conditions $q_0^{\sigma}$ and $q_1^{\sigma}$, where $\sigma \in \prod_{\gamma \in F} 2^{\eta(\gamma)}$ as follows:

        Suppose we have defined $q_{j-1}$, for $j < m$. Since $\dot{x}$ is not added in a proper initial stage of the iteration,

        \[(q_{j-1} \ast \sigma_j) \restriction \delta \Vdash T_{q(\delta) \ast \sigma_j(\delta)^\frown q \restriction (\delta, \alpha)}(\dot{x}) \text{ is a perfect tree}.\]

        Hence, there are names for conditions $q_0^{\sigma_j}$ and $q_1^{\sigma_j}$ such that 

        \[(q_{j-1} \ast \sigma_j) \restriction \delta \Vdash q_i^{\sigma_j} \leq (q(\delta) \ast \sigma_j (\delta)^\frown i)^\frown q \restriction (\delta, \alpha)\]
    
        and

        \[(q_{j-1} \ast \sigma_j) \restriction \delta \Vdash \left(\left[T_{q_{0}^{\sigma_j}} (\dot{x})\right] \times \left[T_{q_1^{\sigma_j}} (\dot{x})\right]\right) \cap G = \emptyset.\]
        
        Let $q_j \leq_{F, \eta} q_{j-1}$ be a condition such that $(q_j \ast \sigma_j) \restriction \delta$ decides all maximal initial segments of $\dot{x}$ that are decided by each $q_{i}^{\sigma_j}$. Finally, let $r \leq_{F, \eta} q_{m-1}$ be such that $r \restriction \delta = q_{m-1} \restriction \delta$ and, for all $\sigma \in \prod_{\gamma \in F} 2^{\eta(\gamma)}$ and all coordinatewise extensions $\sigma' \in \prod_{\gamma \in F} 2^{\eta'(\gamma)}$ of $\sigma$,

        \[(r \ast \sigma') \restriction \delta \Vdash (r \ast \sigma') \restriction [\delta, \alpha) = q_{\sigma'(\eta(\beta))}^\sigma.\]

        Then $r$ is our desired condition. 
        
        \item Case 2: $\alpha = \xi + 1$.
        
        Assume, without loss of generality, $\xi \in F$. Moreover, using Lemma \ref{mainlemma2}, assume
    
        \[q \restriction \xi \Vdash [T_{q(\xi)} (\dot{x})]\text{ is $G$-independent}.\]
        
        We may define a $\leq_{F, \eta}$-decreasing sequence $(q_j)_{j \leq m}$ in $\mathbb{V}_{\alpha}$ as follows: 

        Suppose we have defined $q_{j-1}$, for $j < m$ and assume, by shrinking $q_{j-1}$ if necessary, that $(q_{j-1} \ast \sigma_j) \restriction \xi$ decides all maximal initial segments of $\dot{x}$, that are decided by $q_{j-1} (\xi) \ast \tau$, for each $\tau \in 2^{\eta(\xi)}$.
    
        Let $q_j \leq_{F, \eta} q_{j-1}$ be such that
    
        \[(q_j \ast \sigma_j^\frown i) \restriction \xi \Vdash q_j(\xi) = \bigcup_{\tau \in 2^{\eta(\xi)}} q_{j-1} (\xi) \ast \tau^\smallfrown i.\]
    
        Then $r = q_{m-1}$ ir the desired condition. \qedhere

        \end{itemize}

\end{proof}

Note that in Case 1 we simply used that the graph does not have perfect cliques. In fact reals added at limit stages of iterations of both $\mathbb{E}_0$ and $\mathbb{V}$ are contained in compact $G$-independent sets, coded in the ground model, as long as $G$ is a closed graph without perfect cliques. It is at successor stages where the difference occurs.

\begin{proof}[Proof of Theorem \ref{silverthm}]
    
    Using Lemma \ref{bookkeeping} and some bookkeeping, we may construct a fusion sequence $(p_n)_{n \in \omega}$, witnessed by $(F_n)_{n \in \omega}$ and $(\eta_n)_{n \in \omega}$, such that each $p_n$ is $G$-$(F_n, \eta_n)$-faithful. This way, if $q$ is the fusion of $(p_n)_{n \in \omega}$, then $[T_q (\dot{x})]$ is a $G$-independent set. \qedhere
    
\end{proof}

\section{Forcing with fat Silver trees}

For an analytic graph $G$ on a Polish space $X$, let $I_G$ be the $\sigma$-ideal of analytic sets with countable Borel chromatic number --- i.e., for every $A \in I_G$, $\chi_B (G \cap A^2) \leq \aleph_0$. It follows from the next fact that this ideal has the \textit{inner approximation property} (see 3.2 of \cite{kechris1987structure}) --- i.e., any $I_G$-positive analytic set contains a compact $I_G$-positive set.

\begin{fact}[Theorem 6.3 of \cite{kechris1999borel}]\label{g0dich} Let $X$ be a Polish space and $G$ be an analytic graph on $X$, then exactly one of the following holds:

\begin{itemize}
    \item[(1)] either $\chi_B (G) \leq \aleph_0$; or
    \item[(2)] there is a continuous homomorphism from $G_0$ to $G$.
\end{itemize}

\end{fact}

To see why this implies that $I_G$ has the inner approximation property, for an analytic set $A$, apply the fact above to the graph $G \cap A^2$.  If $\chi_B (G \cap A^2) > \aleph_0$ and $\varphi$ is a continuous homomorphism from $G_0$ to $G$, then $\varphi[2^\omega] \subseteq A$ is the desired compact $I_G$-positive set.  

Surprisingly, the most natural forcing notion to increase $\chi_B (G_0)$ --- the forcing notion of Borel $I_{G_0}$-positive sets --- is \textit{not} proper (see Theorem 4.7.20 of \cite{zapletal2008idealized}). We solve this problem by introducing a \textit{corrected} forcing notion of perfect trees which will also increase $\chi_B (G_0)$ but is, however, proper.

Let $I_{G_0}^t$ be the $\sigma$-ideal of translations of sets in $I_{G_0}$ --- i.e., 
\[A \in I_{G_0}^t \leftrightarrow \exists n \in \omega\ (A + 1_n \in I_{G_0}),\]
where $1_n \in 2^\omega$ is the image of the caracteristic function of ${n}$, and $+$ denotes the coordinatewise sum mod $2$, defined on $2^\omega$.

Clearly, $I_{G_0} \subseteq I_{G_0}^t$ and, therefore, $\cov(I_{G_0}^t) \leq \chi_B (G_0)$.

We draw inspiration from the fact that the Silver forcing is equivalent to the forcing notion of Borel $I_{G_1}$-positive sets (see Fact \ref{zapsilver}) and isolate a forcing notion of Silver trees which is equivalent to the forcing notion of Borel $I_{G_0}^t$-positive sets.

Say that a Silver tree on $2^{<\omega}$ is a \textit{$G_0$-tree} iff

\[\forall s \in \mathrm{spl}(p)\ \forall t \in 2^{|s|}\ \left(([p_s] + t) \times ([p_s] + t) \cap G_0 \neq \emptyset\right).\]

This condition will be referred to as \textit{$G_0$-fatness}. Note that this ``fatness'' property depends on the group action that we are considering on the underlying space --- in our case, the rational shifts on the Cantor space. Let $\mathbb{G}_0$ denote the forcing notion of $G_0$-trees, ordered by inclusion.

The next claim and lemma are slight modifications of Claim 2.3.31 and Lemma 2.3.29 of \cite{zapletal2004descriptive}, respectively.

\begin{claim}\label{fatclaim} Let $A$ be an $I_{G_0}^t$-positive analytic subset of $2^\omega$ and $s \supseteq \mathrm{st}(A)$. Then for every $t \in 2^{|s|}$, there exists some $n_t \in \omega$ and an $I_{G_0}^t$-positive analytic set $A_t \subseteq A$ such that

\begin{itemize}
    \item[(1)] $A_t + 1_{n_t} \subseteq A$; and
    \item[(2)] $s_{n_t} \subseteq \mathrm{st}(A_t + t)$.
\end{itemize}
\end{claim}

\begin{proof} 
    For every $t \in 2^{|s|}$, and every $n \in \omega$, let
    \[A_{t,n} \doteq \{x \in A\ |\ s_{n} \subseteq x + t \text{ and } x + 1_n \in A\}.\]
    
    Since $A \setminus \bigcup_{t \in 2^{|s|}, n \in \omega} A_{t, n}$ is clearly an $I_{G_0}^t$-small analytic set, $A$ is $I_{G_0}^t$-positive and $I^t$ is a $\sigma$-ideal, there exists $n_t$ such that $A_t = A_{t, n_t}$ is $I_{G_0}^t$-positive. It is easy to see that one such set is as required. \qedhere
\end{proof}


\begin{lem} Let $A$ be an analytic subset of $2^\omega$. Then either $A \in I_{G_0}^t$, or it contains the branches of some $G_0$-tree.

\end{lem}

\begin{proof}
    Let $p, p'$ be finite uniform trees --- i.e., for $s, t \in p$ of same length, \[s^\smallfrown i \in p \leftrightarrow t^\smallfrown i \in p,\]
    for each $i < 2$). Say that $p'$ is a \textit{fat-extension} of $p$ iff for every $t \in 2^{<\omega}$ such that $|t| = \mathrm{ht}(p)$, there exists $t' \in 2^{<\omega}$ such that $s_{n_t} + t' \in \mathrm{spl}(p')$,
    where $n_t$ is the least $n$ such that $s_{n_t} \supseteq t$.
    
    Let $T$ be a tree on $(2 \times \omega)^{<\omega}$ such that $A$ is the projection of $T$ onto the first coordinate. Inductively we construct sequences $(p^n)_{n \in \omega}$, of finite binary trees, and $(q^n)_{n \in \omega}$, of finite subtrees of $\omega^{<\omega}$, such that
    
    \begin{itemize}
        \item the endnodes of $p^n$ are $t_0^n, ..., t_{k_n-1}^n$, and of $q^n$ are $u_0^n, ..., u_{k_n-1}^n$;
        
        \item $(t_i^n, u_i^n) \in S$, for $i < k_n$;
    
        \item $p^n$ is a finite uniform tree fat-extending $p^{n-1}$; and
        
        \item $A^n \doteq \bigcap_{i < k_n} \mathrm{proj}[T \restriction (t_i^n, u_i^n)] - t_i^n + t_0^n$ is an $I_{G_0}^t$-positive set.
    
    \end{itemize}
    
    Assume $(t_i^n)_{i \in 2^n}, (u_i^n)_{i < k_n}$ and $A^n$ have been constructed. 
    
    Let $m \doteq |t_0^n|$ and complete the level $2^m$ with $\left\{t_{k_n}^n, ..., t_{2^m - 1}^n\right\}$ --- i.e.,  $2^m = \left\{t_0^n, ..., t_{k_n -1}^n, t_{k_n}^n, ..., t_{2^m - 1}^n\right\}$. We proceed using successive induction in $2^m$ steps:
    
    Assume we have defined $n_{j-1} \in \omega$, $A_{j-1} \notin I_{G_0}^t$ and, finite trees $p^{n_{j-1}}$ and $q^{n_{j-1}}$, with endnodes $\left\{t_0^{n_{j-1}}, ..., t_{\ell_n - 1}^{n_{j-1}}\right\}$ and $\left\{u_0^{n_{j-1}}, ..., u_{\ell_n - 1}^{n_{j-1}}\right\}$, respectively, satisfying the four items above. 
    
    Let $n_j > n_{j-1}$ and $A_j \subseteq A_{j-1}$ be as in Claim \ref{fatclaim} --- i.e.,  $A_j + 1_{n_j} \subseteq A_{j-1}$; and $s_{n_j} \subseteq \mathrm{st}(A_j + t_j^n)$. For every $i < \ell_n^{n_{j-1}}$,  let $A_j \ast t^{n_{j-1}}_i \subseteq A_{j-1}$ denote the copy of $A_j$, inside $A_{j-1}$ above $t_{i}^{n_{j-1}}$ and let $t^{n_j}_{2i} = \mathrm{st} (A_j \ast t^{n_{j-1}}_i)$ and $t^{n_j}_{2i+1} = \mathrm{st}(A_j \ast t^{n_{j-1}}_i + 1_{n_j})$.  
    Moreover, for each $i < \ell_n^{n_{j-1}}$, choose $u_i^{n_{j-1}}$ such that $(t_i^{n_{j-1}}, u_i^{n_{j-1}}) \in T$ and the set $A^{n_j}$ is $I_{G_0}^t$-positive. Now $p^{n_j}$ is the finite tree generated by downward closure of the nodes $t_i^{n_j}$, for $i < \ell_n^{n_{j-1}}$, and similarly for $q^{n_j}$ (with $u$'s instead of $t$'s).
    
    Let $p^{n+1} = p^{n_{(2^m - 1)}}$ and $q^{n+1} = q^{n_{(2^m - 1)}}$ and note that $p^{n+1}$ satisfies both $G_0$-fatness and Silverness. Then
    \[p = \bigcup_{n \in \omega} p^n\]
    is the desired fat $G_0$-tree --- i.e., $[p] \subseteq A$. \qedhere
\end{proof}

The proof of the theorem above already hints a proof of properness for $\mathbb{G}_0$ and, in fact, $\omega^\omega$-boundedness (which implies the preservation of $\mathfrak{d}$). We soon will see that we are actually not far from proving the \textit{Sacks property} as well.

\begin{lem} $\mathbb{G}_0$ satisfies strong Axiom A. 
\end{lem}

\begin{proof} First, for each $n \in \omega$, we build a finite subtree $p^n$ of $p$, with a set of terminal nodes $L_n (p)$, as follows:

Let $p^0 = \{\mathrm{st}(p)\}$ and assume we have defined $p^{n}$ with terminal nodes $L_n (p) = \left\{t_0, ..., t_{k_n - 1}\right\}$, where $\mathrm{ht}(p^n) = |t_0| = ... = |t_{k_n} - 1| \doteq \ell_n$. For each $t \in 2^{\ell_n}$, let $n_t$ be the least natural number for which $s_{n_t}$ is a splitting node of $(p_{t_0}) + t$. Let $p^{n+1}$ be the finite tree generated by the splitting nodes of $p^n$, together with all the nodes of $p$ of height $n_t$, for $t \in 2^{\ell_n}$. Note that \[|L_{n+1} (p)| = |L_n (p)| \cdot 2^{2^{\mathrm{ht}\left(p^n\right)}}.\]
Finally, for $q \in \mathbb{G}_0$,

\[q \leq_n p \leftrightarrow q \leq p \text{ and } q^n = p^n.\]

The item (1) of the statement of the strong Axiom A is easily seen to be satisfied. As for (2), let $L_n (p) = \left\{t_0, ..., t_{k_n-1}\right\}$. We shall define a finite sequence $(q_i)_{i < k_n}$ using successive amalgamation:

Assume we have defined $q_{i-1} \leq_n p$ such that $q_{i-1} \ast \sigma_j$ is compatible with at most one element of $A$, for every for $i < k_n$ and $j < i$. Let $r_i \leq q_{i-1} \ast \sigma_i$ be a condition compatible with at most one element of $A$ and define $q_i$ to be the amalgamation of $r_i$ into $q_{i-1}$. Now it is easy to see that $q = q_{k_n-1}$ is such that $q \leq_n p$ and it is compatible with at most $k_n$ elements of $A$. \qedhere

\end{proof}

\begin{cor} It is consistent with ZFC that $\mathfrak{d} < \chi_B (G_0)$. 
\end{cor}

Now let us see what happens to $\cov(\mathcal{N})$. Of course a natural way of showing that this forcing notion does not add random reals is by proving the \textit{Laver property}. However, since this forcing notion is $\omega^\omega$-bounding, the junction of these properties --- the Sacks property --- gives us something stronger. Namely, $\cof(\mathcal{N})$ --- the least cardinality of a basis for $\mathcal{N}$, and the biggest cardinal at the \textit{Cichoń's diagram} --- has the value of the ground-model continuum.

For a function $f \in \omega^\omega$, recall that an \textit{$f$-slalom} is a function $S: \omega \rightarrow [\omega]^{<\omega}$ such that $|S(n)| \leq f(n)$, for all $n \in \omega$. An element $x \in \omega^\omega$ is covered by $S$ iff $x(n) \in S(n)$, for almost all $n \in \omega$. Due to a characterization of Bartoszy\'{n}ski (see \cite{bartoszynski1995set}), $\cof(\mathcal{N})$ is the least cardinality of a family of $f$-slaloms whose union covers $\omega^\omega$, for some $f \in \omega^\omega$.

Finally, say that a forcing notion $\mathbb{P}$ has the \textit{Sacks property} iff every element of $\omega^\omega$, in the generic extension, is covered by a ground-model $f$-slalom. Clearly, forcing notions with the Sacks property do not increase $\cof(\mathcal{N})$.

\begin{thm} $\mathbb{G}_0$ has the Sacks property.
\end{thm}

The next colorally follows from the properness of $\mathbb{G}_0$, together with the preservation of the Sacks property under countable support iterations of proper forcing notions.

\begin{cor} It is consistent with ZFC that $\cof(\mathcal{N}) < \chi_B (G_0)$.
\end{cor}

\section{Questions}

As said in the Introduction, the consistency of $\chi_B (G_0) < \cov(\mathcal{N})$ is open. The same is true for $\chi_B (G_0) < \mathfrak{h}$.

\begin{quest}\label{questrandom} Is it consistent with ZFC that $\chi_B (G_0) < \cov(\mathcal{N})$? What about $\chi_B (G_0) < \mathfrak{h}$?
\end{quest}

This question really tests the role of minimality. It is well-known that the random forcing increases $\cov(\mathcal{N})$ and Mathias forcing increases $\mathfrak{h}$. Note that both random reals and Mathias reals are contained in ground-model Borel $G_0$-independent sets. The argument for the Mathias forcing is similar to the Silver. As for the random forcing: 

First, use Theorem 3.3 of \cite{miller2008measurable} to obtain three measurable $G_0$-independent sets $A_0, A_1$ and $A_2$ such that $2^\omega = A_0 \cup A_1 \cup A_2$. Now, if $p$ is a positive-measure Borel set, then for some $i < 3$, the set $p \cap M_i$ is positive-measure as well, and contains some positive-measure Borel set $q$. Hence, the random real is contained in a $G_0$-independent Borel set coded in the ground model. However, since the random extension is far from minimal, we do not know what happens to other reals. 

\begin{quest} Is it consistent with ZFC that $\chi_B (G_1) < \mathfrak{b}$? What about $\chi_B (E_0) < \mathfrak{b}$?
\end{quest}

The inequality $\chi_B (G_1) < \mathfrak{d}$ holds in the Miller model, simply because $\mathfrak{r} < \mathfrak{d}$ holds already (recall that $\chi_B (G_1) \leq \mathfrak{r}$), which follows from the preservation of $p$-points. It is likely that a proof of $\chi_B (G_1) < \mathfrak{b}$ would require a more direct argument for Laver iterations.

This question is also related to the long-standing question on the differences between Laver and Silver notions of \textit{measurability} (see \cite{silvermeasurability}).

Last, it is not clear that the forcing $\mathbb{G}_0$ satisfies the \textit{$2$-localization property} --- a stronger version of the Sacks property. We say that $\mathbb{P}$ has the $2$-localization property iff every element of $\omega^\omega$, of the generic extension, is contained in a ground-model \textit{binary} tree --- i.e., a tree on $\omega^{<\omega}$ such that every node has at most two successor nodes. 

In some cases a proof of minimality yields a proof of the $2$-localization property. This is the case, for instance, with the Sacks or Silver forcing notions. In some cases, the forcing notion adds reals of minimal degree despite, not having the $2$-localization property, as it is the case with the Laver forcing. We do not know, however, whether $\mathbb{G}_0$ has the $2$-localization property or adds reals of minimal degree. 

\begin{quest} Does $\mathbb{G}_0$ have the $2$-localization property? Does it add reals of minimal degree? 
\end{quest}

It is not clear to the authors whether the methods from \cite{grigorieff1971combinatorics} can be used for this forcing notion.

\bibliographystyle{plain}
\bibliography{bibliography}

\end{document}